\documentclass{amsart} 

\usepackage{amsmath,amssymb,amsthm,amsfonts,amsxtra,mathtools}
\usepackage{enumerate,verbatim,mathrsfs,comment,color,url}
\usepackage[all,2cell,ps]{xy}
\usepackage[pagebackref]{hyperref}
\usepackage{graphicx}
\usepackage{tikz-cd}
\usepackage{cleveref}
\usepackage{verbatim}
\usepackage{comment}

\newcommand{\mycomment}[1]{}

\DeclareMathOperator{\ann}{ann}

\DeclareMathOperator{\cid}{CI-dim}
\DeclareMathOperator{\cidim}{CI-dim}

\DeclareMathOperator{\codim}{codim}

\DeclareMathOperator{\coker}{Coker}
\DeclareMathOperator{\cone}{Cone}
\DeclareMathOperator{\curv}{curv}
\DeclareMathOperator{\cx}{cx}
\DeclareMathOperator{\depth}{depth}

\DeclareMathOperator{\Ext}{Ext}

\DeclareMathOperator{\grade}{grade}

\DeclareMathOperator{\Hom}{Hom}
\DeclareMathOperator{\htt}{ht}
\DeclareMathOperator{\id}{id}
\DeclareMathOperator{\Image}{Image}

\DeclareMathOperator{\injcx}{inj\,cx}
\DeclareMathOperator{\injcurv}{inj\,curv}

\DeclareMathOperator{\Ker}{Ker}

\DeclareMathOperator{\pd}{pd}

\DeclareMathOperator{\px}{px}

\DeclareMathOperator{\rank}{rank}

\DeclareMathOperator{\Tor}{Tor}

\DeclareMathOperator{\Tr}{Tr}
\DeclareMathOperator{\type}{type}

\renewcommand{\ge}{\geqslant}
\renewcommand{\le}{\leqslant}

\newcommand{\fm}{\mathfrak{m}}

\newcommand{\fa}{\mathfrak{a}}
\renewcommand{\iff}{if and only if }
\newcommand{\lra}{\longrightarrow}

\theoremstyle{plain}
\newtheorem{theorem}{Theorem}[section]
\newtheorem{lemma}[theorem]{Lemma}
\newtheorem{proposition}[theorem]{Proposition}
\newtheorem{corollary}[theorem]{Corollary}
\newtheorem{remark}[theorem]{Remark}

\theoremstyle{definition}
\newtheorem{definition}[theorem]{Definition}

\newtheorem{example}[theorem]{Example}

\newtheorem{para}[theorem]{}
\newtheorem{question}[theorem]{Question}
\newtheorem{setup}[theorem]{Setup}

\theoremstyle{remark}

\numberwithin{equation}{section}

\title[Homological invariants of modules under linkage]{Complexity, curvature and homological dimension of modules under linkage}

\author[S.~Bhowmick]{Subhadip Bhowmick}
\address{Department of Mathematics, Indian Institute of Technology Kharagpur, West Bengal - 721302, India}
\email{sbhowmick712@kgpian.iitkgp.ac.in, subhadipbhowmick712@gmail.com}

\author[D.~Ghosh]{Dipankar Ghosh}
\address{Department of Mathematics, Indian Institute of Technology Kharagpur, West Bengal - 721302, India}
\email{dipankar@maths.iitkgp.ac.in, dipug23@gmail.com}
\urladdr{\url{https://orcid.org/0000-0002-3773-4003}}

\date{\today}

\begin{document}

\pagenumbering{arabic}
\thispagestyle{empty}
\begin{abstract}
In this article, we analyze how (projective and injective) complexity, curvature, and complete intersection dimension behave under linkage of modules and ideals. Let $R$ be a Gorenstein local ring. Consider a Gorenstein perfect ideal $\fa$ (e.g., $\fa$ is generated by an $R$-regular sequence). Let $M$ and $N$ be two Cohen-Macaulay $R$-modules linked by $\fa$. We prove that $\cx_R(M) = \injcx_R(N)$ and $\curv_R(M) = \injcurv_R(N)$. In particular, when $R$ is complete intersection, $\cx_R(M) = \cx_R(N)$ and $\curv_R(M) = \curv_R(N)$. Furthermore, we show that $\pd_R(M)=\pd_R(N)$ and $\cid_R(M)=\cid_R(N)$. If any of these dimensions is finite, it is equal to $\htt(\fa)$. Similar results are obtained for linkage of ideals. All these results highly extend a classical result of Peskine and Szpiro in many directions. We construct several examples that complement our results. These also show how properties like `integrally closed', `$\fm$-full' and `Burch' behave under linkage of ideals.
\end{abstract}
\maketitle

\section{Introduction}
In order to investigate algebraic curves in $\mathbb{P}^3$, Noether, Halphen, and Severi developed the theory of linkage (or liaison) of algebraic varieties in the late 19th and early 20th centuries. Linkage makes it possible to shift from one curve to another that is geometrically connected. By repeated application of this process, a sequence of interconnected curves is formed, known as a linkage class. Linkage plays an important role as it preserves specific characteristics of the curve, which can make the resultant curves easier to handle than the original. In 1974, Peskine and Szpiro reduced the general concept of linkage to algebraic considerations involving specific ideals in regular local rings in their paper \cite{PS74}, which established an excellent algebraic foundation for linkage theory. Some authors, such as Martin \cite{Mar00}, Yoshino and Isogawa \cite{YI00}, Martsinkovsky and Strooker \cite{MS04}, and Nagel \cite{Nag05}, applied the theory of linkage to module configuration in various ways. Several research works on linkage theory in the context of modules have been executed applying these generalizations; see, for example, \cite{DS13, DS15, IT16}.


Let $R$ be a commutative Noetherian local ring. Two ideals $I$ and $J$ of $R$ are said to be linked by an ideal $\fa$ if $\fa\subseteq I\cap J$, $I= (\fa:_R J) $ and $J= (\fa:_R I)$. Some of the main results by Peskine-Szpiro in the theory of linkage of ideals are the following.

\begin{theorem}[Peskine-Szpiro \cite{PS74}]\label{thm:Peskine-Szpiro}
    Let $R$ be a Gorenstein local ring. Let $I$ and $J$ be two ideals of $R$ linked by an $R$-regular sequence. Then
    \begin{enumerate}[\rm (1)]
        \item $R/I$ is CM $($Cohen-Macaulay$)$ \iff $R/J$ is CM.
        \item  If $R/I~($equivalently, $R/J)$ is CM, then
        \[
            \pd_R(R/I) < \infty \; \Longleftrightarrow \; \pd_R(R/J) < \infty.
        \]
    \end{enumerate}
\end{theorem}


The notions of (projective and injective) complexity and curvature of a module were introduced by Avramov in \cite{Avr89a}, \cite{Avr89b} and \cite{Avr96} to distinguish modules of infinite homological dimension. They measure respectively the polynomial and exponential growth rate of Betti and Bass numbers of a module, see Definition~\ref{defn:cx-injcx}. The projective complexity and curvature of an $R$-module $M$ are denoted by $\cx_R(M)$ and $\curv_R(M)$ respectively. In the literature, these are simply called complexity and curvature of $M$. For injective complexity and curvature of $M$, we use the notations $\injcx_R(M)$ and $\injcurv_R(M)$, respectively. Note that any of $\cx_R(M)$ and $\curv_R(M)$ (resp., $\injcx_R(M)$ and $\injcurv_R(M)$) is zero \iff the projective (resp., injective) dimension of $M$ is finite; see Remark~\ref{remark}.

Using the concept of projective dimension, Avramov, Gasharov, and Peeva \cite{AGP97} introduced another homological dimension to characterize complete intersection local rings. This is called complete intersection dimension, in short CI-dimension.

In this paper, we are interested in the linkage of modules due to Martsinkovsky and Strooker \cite{MS04}.
They extended the notion of linkage from ideals to modules by introducing the operator $\lambda_R(-):= \Omega\Tr(-)$, where $\Tr(M)$ and $\Omega(M)$ respectively denote the Auslander transpose and syzygy of an $R$-module $M$, see \ref{para:syz} and \ref{para:Aus-tr}. According to the definition in \cite{MS04}, two $R$-modules $M$ and $N$ are said to be linked by an ideal $\fa$ if $\fa\subseteq \ann_R(M)\cap\ann_R(N)$, $M\cong\lambda_{R/\fa}(N)$ and $N\cong\lambda_{R/\fa}(M)$, see Definition~\ref{defn:linked-modules}. It follows that two proper ideals $I$ and $J$ are linked by $\fa$ as ideals of $R$ if and only if $R/I$ and $R/J$ are linked by $\fa$ as $R$-modules, cf.~Remark~\ref{rmk:linked-ideal-red}.

The main aim of this article is to analyze how projective (resp., injective) complexity and curvature, as well as CI-dimension, behave under linkage of modules (and hence of ideals) by a Gorenstein perfect ideal over a Gorenstein local ring. A particular class of Gorenstein perfect ideals is formed by complete intersection ideals (that is, ideals generated by regular sequences), which provide the classical setting of linkage theory. The behavior of complexity, curvature, and CI-dimension under linkage was not known even in this classical setting. We prove the following results.


\begin{theorem}\label{thm:main}
 Let $R$ be a Gorenstein local ring. Let $M$ and $N$ be two finitely generated $R$-modules linked by a Gorenstein perfect ideal $\fa$ $($e.g., $\fa$ is generated by an $R$-regular sequence$)$; {\rm cf.~\ref{para:Gor-ideal}}. Suppose that $M$ is CM. Then the following hold.
    \begin{enumerate}[\rm (1)]
        \item $\cx_R(M) = \injcx_R(N)$ and $\curv_R(M) = \injcurv_R(N)$.
        \item When $R$ is complete intersection, $\cx_R(M) = \cx_R(N)$ and $\curv_R(M) = \curv_R(N)$.
        \item $\pd_R(M)=\pd_R(N)$. When they are finite, they are equal to $\htt(\fa)$.
        \item $\cid_R(M)=\cid_R(N)$. When they are finite, they are equal to $\htt(\fa)$.
    \end{enumerate}   
\end{theorem}

Theorem~\ref{thm:main}.(1) in particular ensures that both projective and injective complexity (resp., curvature) of CM modules are preserved under even linkage by Gorenstein perfect ideals in a Gorenstein local ring, while Theorem~\ref{thm:main}.(2) shows that all these four invariants of CM modules are preserved under linkage by Gorenstein perfect ideals in a complete intersection local ring.

Theorem~\ref{thm:main}.(3) can be derived from \cite[Prop.~10 and Cor.~15]{MS04}. As it is not specifically mentioned anywhere in \cite{MS04}, we note it here as a consequence of Theorem~\ref{thm:main}.(1) using the fact that over a Gorenstein local ring, a finitely generated module has finite projective dimension \iff it has finite injective dimension. It follows from \cite[p.~608, Thm.~4]{MS04} that Theorem~\ref{thm:main}.(3) is not true in general for modules which are not CM. Thus, the assertions in Theorem~\ref{thm:main} do not hold in general if $M$ is not CM.

It is worth noting separately the version of Theorem~\ref{thm:main} for linkage of ideals. We obtain this as a corollary of Theorem~\ref{thm:main}.

\begin{corollary}\label{cor:main}
      Let $R$ be a Gorenstein local ring. Let $I$ and $J$ be two proper ideals of $R$ linked by a Gorenstein perfect ideal $\fa$ $($e.g., $\fa$ is generated by an $R$-regular sequence$)$. Suppose that $R/I$ is CM. Then the following hold.
     \begin{enumerate}[\rm(1)]
        \item 
        $\cx_R(I) = \injcx_R(J)$ and $\curv_R(I) = \injcurv_R(J)$.
        \item 
        When $R$ is complete intersection, $\cx_R(I) = \cx_R(J)$ and $\curv_R(I) = \curv_R(J)$.
        \item 
        $\pd_R(R/I)=\pd_R(R/J)$. When they are finite, they are equal to $\htt(\fa)$.
        \item 
        $\cid_R(R/I)=\cid_R(R/J)$. When they are finite, they are equal to $\htt(\fa)$.
    \end{enumerate} 
\end{corollary}

In view of Remark~\ref{remark}, Corollary~\ref{cor:main}.(1) highly strengthens Theorem~\ref{thm:Peskine-Szpiro}.(2), while Corollary~\ref{cor:main}.(4) provides the counterpart of Theorem~\ref{thm:Peskine-Szpiro}.(2) for CI-dimension.

There has been a surge of interest to find classes of modules having maximum possible complexity or curvature. See \ref{para:max-cx-curv} and \ref{para:max-cx-curv-survey} for the terminology and a detailed survey on this. As a consequence of Theorem~\ref{thm:main}.(1), one gets that if an $R$-module $M$ is Ulrich (Definition~\ref{defn:Ulrich-Burch}.\eqref{defn:Ulrich}), or if $M$ is Burch (Definition~\ref{defn:Ulrich-Burch}.\eqref{defn:Burch}), then with the hypotheses as in Theorem~\ref{thm:main}, the module $N$ has maximal projective (resp., injective) complexity and curvature; see Corollary~\ref{cor:Burch} for more results in this direction. Furthermore, in a Gorenstein local ring $(R,\fm)$, an $\fm$-primary ideal, which is linked to an integrally closed ideal or a Burch ideal, also has maximal projective (resp., injective) complexity and curvature; cf.~Corollary~\ref{cor:Burch-ideal}. For reader's information, we note some classes of Burch modules and ideals in Example~\ref{exam:Burch}.

We now describe the contents of the article. In Section~\ref{sec:prel}, we provide some notations, definitions and preliminary results that are used in the article. In Section~\ref{sec:main-results}, we prepare some technical lemmas, which lead to the proofs of Theorem~\ref{thm:main} and Corollary~\ref{cor:main}. Finally, in Section~\ref{sec:applications}, we show some applications of Theorem~\ref{thm:main}, which are Corollaries~\ref{cor:Burch} and \ref{cor:Burch-ideal}. In the same section, we construct several examples that complement our results. We observe that the Ulrich property of modules is preserved under linkage by a Gorenstein ideal, and it is not true in general; see Proposition~\ref{prop:Ulrich-linkage} and Example~\ref{exam:Ulrich-is-not-preserved}. The assertions (1) and (2) in Theorem~\ref{thm:main} do not hold in general if we omit the condition that $\fa$ is perfect, see Example~\ref{exam:Burch-m-full-is-not-preserved}. The properties `integrally closed' and `$\fm$-full' are not always preserved under linkage of ideals by a regular sequence even in a regular local ring as shown in Example~\ref{exam:int-closed-is-not-preserved}.


\section{Preliminaries}\label{sec:prel}

 In this section, we recall some terminologies, definitions, and basic properties, which are used later in the article. {\it Throughout, $R$ is a commutative Noethenian local ring with the maximal ideal $\fm$ and the residue field $k$. All modules over $R$ are assumed to be finitely generated.}
 Let $M$ be an $R$-module. Set $ M^* := \Hom_R(M,R) $. The multiplicity of $M$ is denoted by $e(M)$, while $\mu(M)$ denotes the minimal number of generators of $M$. The codimension of $M$ is defined to be $\codim(M) := \dim(R) - \dim(M)$, while codimension of the ring $R$ is given by $\mu(\fm)-\dim(R)$. Let $\beta_n^R (M) := \rank_k(\Ext^n_R(M,k))$ denote the $n$-th Betti number of $M$, while the $n$-th Bass number of $M$ is denoted by $\mu_R^n(M) := \rank_k(\Ext^n_R(k,M))$. Set $\type(M):=\mu_R^t(M)$, where $t=\depth(M)$. The annihilator of $M$ is denoted by $\ann_R(M):=\{ r \in R : rM=0 \}$. The abbreviations CM and MCM are used for Cohen-Macaulay and maximal Cohen-Macaulay respectively.
 
 \begin{para}\label{para:syz}
      For $n\ge 0$, let $\Omega^R_n(M)$ denote the $n$-th syzygy module of $M$ in a minimal free resolution of $M$, i.e., if $\dots \rightarrow F_1\xlongrightarrow{d_1}F_0\xlongrightarrow{d_0}M \xlongrightarrow{d_{-1}} 0$ is a minimal free resolution of $M$ (along with the augmentation), then $\Omega^R_n(M) := \coker d_{n+1}=\Ker d_{n-1}$. Sometimes, we write $\Omega(M)$ to denote $\Omega^R_1(M)$. Note that in a minimal free resolution, for each $n\ge 0$, $\Omega^R_n(M)$ is unique upto isomorphism.
 \end{para}
 
 \begin{para}\label{para:Aus-tr}\cite[Def.~2.5]{AB69}
     Let $F_1\xlongrightarrow{d_1}F_0\xlongrightarrow{d_0}M \rightarrow 0$ be a free presentation of $M$. The (Auslander) transpose of $M$ is defined to be $\Tr(M) := \coker({d_1^*})$. It induces an exact sequence \begin{equation*}
          0\rightarrow M^* \xlongrightarrow{d_0^*}{F_0^*}\xlongrightarrow{d_1^*}{F_1^*}\rightarrow \Tr(M) \rightarrow 0.
       \end{equation*}
       If the free presentation of $M$ is minimal, then $\Tr(M)$ is unique up to isomorphism, and ${F_0^*}\xlongrightarrow{{d_1^*}}{F_1^*}\rightarrow\Tr(M) \rightarrow 0$ is a minimal free presentation of $\Tr(M)$. Set
       $$\lambda(M):= \Omega\Tr(M).$$
 \end{para}

 \begin{para}\label{para:Gor-ideal}
     An ideal $\fa$ of $R$ is said to be Gorenstein if $R/\fa$ is a Gorenstein ring. Note that if $\fa$ is a Gorenstein ideal of $R$, then $R/\fa$ is CM as an $R$-module. In addition, if $\pd_R{(\fa)}$ is finite, then $\fa$ is a perfect ideal, i.e., $R/\fa$ is a perfect $R$-module.
 \end{para}
 \begin{definition}\label{defn:linked-ideals}\cite[Sec.~2]{PS74}
     Two ideals $I$ and $J$ of $R$ are said to be linked by an ideal $\fa$ if $\fa\subseteq I\cap J$, $I= (\fa:_R J) $ and $J= (\fa:_R I)$.
 \end{definition}

 Martsinkovsky and Strooker extended the notion of linkage from ideals to modules as follows.
 
 \begin{definition}\label{defn:linked-modules}
 \begin{enumerate}[\rm (1)]
     \item
     \cite[Def.~3]{MS04} Two $R$-modules $M$ and $N$ are said to be horizontally linked if $M\cong\lambda(N)$ and $N\cong\lambda(M)$. Thus, an $R$-module $M$ is called horizontally linked (to $\lambda(M)$) if $M\cong\lambda^2(M)$, where $\lambda^2=\lambda\circ\lambda $.
     \item 
     \cite[Def.~4]{MS04} The $R$-modules $M$ and $N$ are said to be linked by an ideal $\fa$ if $\fa\subseteq \ann_R(M)\cap\ann_R(N)$, and $M, N$ are horizontally linked as $R/\fa$-modules.
 \end{enumerate}
 \end{definition}
 
 \begin{remark}\label{rmk:linked-ideal-red}
     It follows from the definition that $I$ and $J$ are linked by an ideal $\fa\subseteq I \cap J$ \iff the ideals $I/\fa$ and $J/\fa$ are linked by the zero ideal of the ring $R/\fa$. When both $I$ and $J$ are proper ideals, the second statement is equivalent to the fact that $R/I$ and $R/J$ are horizontally linked as $R/\fa$-modules, cf.~\cite[Prop.~1 and Def.~2]{MS04}.
 \end{remark}
 
 
\begin{definition}\label{defn:cx-injcx}
\begin{enumerate}[\rm (1)] 
    \item \cite[Def.~1.1]{Avr89a} 
    For an $R$-module $M$, the complexity (more precisely, projective complexity) $\cx_R(M)$ of $M$ is the least non-negative integer $b$ such that there exists $\alpha>0$ satisfying $\beta_n^R(M) \le \alpha n^{b-1}$ for all $n\gg 0$. If no such $b$ exists, then $\cx_R(M):=\infty$.
   \item  \cite[Def.~5.1]{Avr89b}
   Replacing $\beta_n^R(M)$ by $\mu_R^n(M)$ in (1), one defines the injective complexity $\injcx_R(M)$ of $M$. In \cite{Avr89b}, it is called plexity, and denoted by $\px_R(M)$.
   \item 
   \cite[Sec.~1]{Avr96} 
   The (projective) curvature and the injective curvature of $M$, denoted by $\curv_R(M)$ and $\injcurv_R(M)$ respectively, are defined to be
   \begin{equation*}
       \curv_R(M) := \limsup\limits_{n\to\infty}\sqrt[n]{\beta_n^R(M)} \mbox{ and } \injcurv_R(M) := \limsup\limits_{n\to\infty}\sqrt[n]{\mu_R^n(M)}.
   \end{equation*}
   \end{enumerate}
\end{definition}

\begin{remark}\label{remark}
 Let $M$ be an $R$-module. Then
\begin{enumerate}[\rm (1)]
    \item $\pd_R(M) < \infty \Longleftrightarrow \cx_R(M) = 0 \Longleftrightarrow \curv_R(M) = 0 \Longleftrightarrow \curv_R(M) < 1$.
   \item $\id_R(M) < \infty \Longleftrightarrow \injcx_R(M) = 0 \Longleftrightarrow \injcurv_R(M) = 0 \\
   \hspace*{5.18cm}\Longleftrightarrow
   \injcurv_R(M) < 1$.
   \item $\cx_R(M) < \infty \Longrightarrow \curv_R(M) \le 1 \Longleftrightarrow \curv_R(M) =0 \mbox{ or } 1$. Similar implications hold for injective complexity and curvature.
\end{enumerate}
\end{remark}

\begin{proof}
    (1) and (2) can be found in \cite[Sec.~1]{Avr96} and \cite[Rmk.~2.3]{DGS24}, while (3) follows from (1), (2) and the definition of (injective) complexity and curvature.
\end{proof}

The following remark is straightforward and may already be known to experts, so we omit the proof.

\begin{remark}\label{rmk:ses-cx}
 Let $0\rightarrow M \rightarrow K \rightarrow N \rightarrow 0$ be a short exact sequence of $R$-modules. The following statements hold.
  \begin{enumerate}[\rm (1)]
      \item If $\pd_R(K)< \infty$, then $\cx_R(M)=\cx_R(N)$ and $\curv_R(M)=\curv_R(N)$.
       \item If $\id_R(K)< \infty$, then
       $$\injcx_R(M)=\injcx_R(N) \ \mbox{ and } \ \injcurv_R(M)=\injcurv_R(N).$$
  \end{enumerate}
\end{remark}



The notion of CI-dimension is due to Avramov-Gasharov-Peeva\cite{AGP97}
\begin{para}\label{CI-dim}\cite[(1.1) and (1.2), p.~70]{AGP97}
A (codimension $c$) quasi-deformation of $R$ is a diagram of local ring homomorphisms $ R\rightarrow R' \leftarrow S$, where $R\rightarrow R' $ is flat, and $R' \leftarrow S$ is a (codimension $c$) deformation, i.e., $ R' \leftarrow S $ is surjective with kernel generated by a (length $c$) regular sequence.

Set $\cidim_R(0):=0$. When $M \neq 0$, CI-dimension of $M$ is defined to be
\[
  \cidim_R(M) = \inf\{\pd_S(M \otimes_R R')-\pd_S(R') : R\rightarrow R'\leftarrow S \mbox{ is a quasi-deformation} \}.
\]
\end{para}
\begin{para}\label{rem}
    Let $(\mathbb{F}, d^F)$ and $(\mathbb{K}, d^K)$ be two chain complexes, and ${\Phi}: \mathbb{K}\longrightarrow\mathbb{F}$ be a chain map. Dualizing this, one gets a cochain map ${\Phi^*}: \mathbb{F^*}\longrightarrow\mathbb{K^*}$. Consider the mapping cones of $\Phi$  and $\Phi^*$. Their $n$th components are $\cone({\Phi})_n = K_{n-1} \oplus F_n$ and $\cone(\Phi^*)^n = (\mathbb{F}^*)^{n+1} \oplus (\mathbb{K}^*)^n = (F_{n+1})^* \oplus (K_n)^*$, while the differentials are
    \[
    d^{\cone(\Phi)} = \begin{bmatrix}
    -d^K & 0 \\
    -\Phi & d^F 
    \end{bmatrix} \quad \mbox{and} \quad
    d_{\cone(\Phi^*)} = \begin{bmatrix}
    -(d^F)^* &  0 \\
    -\Phi^* & (d^K)^* 
    \end{bmatrix}
    \]
    respectively, cf.~\cite[1.5.1]{We94}. More precisely, the differentials of $\cone(\Phi^*)$ are
    \begin{equation}\label{diff-cone-phi-star}
    d^{n}_{\cone(\Phi^*)} = \begin{bmatrix}
    -(d^F_{n+2})^* &  0 \\
    -(\Phi_{n+1})^* & (d^K_{n+1})^* 
    \end{bmatrix} : \cone(\Phi^*)^n \to \cone(\Phi^*)^{n+1}.
    \end{equation}
    Changing the signs of the differentials of $\mathbb{K}$ and $\mathbb{F}$, ${\Phi}: (\mathbb{K}, -d^K)\longrightarrow(\mathbb{F}, -d^F)$ is also a chain map, whose mapping cone is $\cone({\Phi}) = \mathbb{K}[-1] \oplus \mathbb{F}$ with the differential
    \[
    d^{\cone(\Phi)} = \begin{bmatrix}
    d^K & 0 \\
    -\Phi & -d^F 
    \end{bmatrix}.
    \]
    We rewrite this chain complex as $\cone(\Phi) = \mathbb{F}\oplus \mathbb{K}[-1]$, where the differentials are
    \begin{equation}\label{diff-cone-phi}
    \begin{bmatrix}
    -d^F & -\Phi \\
    0 & d^K 
    \end{bmatrix}, \mbox{ i.e., } d^{\cone(\Phi)}_n = \begin{bmatrix}
    -d^F_n & -\Phi_{n-1} \\
    0 & d^K_{n-1} 
    \end{bmatrix} : \cone(\Phi)_n \to \cone(\Phi)_{n-1}.
    \end{equation}
    Dualizing \eqref{diff-cone-phi-star}, and comparing it with \eqref{diff-cone-phi}, one has $d^{\cone(\Phi)}_n= \Big(d^{n-2}_{\cone(\Phi^*)}\Big)^*$.
\end{para}


\section{Proofs of the main results}\label{sec:main-results}

We first prepare a number of lemmas to prove our main results.

\begin{lemma}\label{lemma}
  Let $R$ be a Gorenstein local ring. Let $F$ and $G$ be free $R$-modules. Let $\xi: F\rightarrow G$ be an $R$-module homomorphism such that $\coker\xi$ is  MCM. Then there exists an exact sequence
  \begin{equation*}
            0 \rightarrow (\coker \xi)^*\longrightarrow G^*\xlongrightarrow{\xi^*} F^*\longrightarrow
           (\Ker\xi)^* \rightarrow 0.
    \end{equation*}
    Moreover, $\Image(\xi^*) \cong (\Image\xi)^*$.
\end{lemma}

\begin{proof}
The homomorphism $\xi$ induces a commutative diagram of exact sequences: 
\begin{equation}\label{diagram}    
\xymatrixrowsep{1.2em}
\xymatrixcolsep{0.8em}
\xymatrix{
0 \ar[rr] && \ker\xi \ar[rr] && F \ar[rr]^{\xi} \ar@{->>}[rd]^{\pi} && G \ar[rr] && \coker{\xi} \ar[rr] && 0\\
       &    &&     && \Image{\xi} \ar@{^{(}->}[ru]^i.
}
 \end{equation} 
Since $\coker\xi$ is MCM, and $F, G$ are free $R$-modules, by depth lemma, $\Image\xi$ and $\Ker\xi$ are MCM.
Therefore, $0 \rightarrow \Ker\xi \rightarrow F \rightarrow \Image\xi \rightarrow 0 $ and $0\rightarrow \Image\xi \rightarrow G\rightarrow \coker\xi\rightarrow 0 $ are exact sequences of MCM modules. It follows that their dual (with respect to $R$) sequences are also exact. Hence one has a commutative diagram of exact sequences:
\begin{equation}\label{dual}
\xymatrixrowsep{1.2em}
\xymatrixcolsep{0.7em}
\xymatrix{
0 \ar[rr] && (\coker \xi)^* \ar[rr] && G^* \ar[rr]^{\xi^*} \ar@{->>}[rd]^{i^*} && F^* \ar[rr] && (\Ker\xi)^* \ar[rr] && 0\\
  &    &&     && (\Image{\xi})^* \ar@{^{(}->}[ru]^{\pi^*}.
}
 \end{equation}  
 Note that the commutative diagram \eqref{dual} is the dual of \eqref{diagram}. By the functorial property of $ \Hom_R(-,R)$, one gets $\xi^* = \pi^*\circ i^*$. Therefore, $\Image(\xi^*) = \Image(\pi^*)$ as $i^*$ is onto. Since $\pi^*$ is injective, it follows that $\Image(\xi^*) \cong (\Image\xi)^*$.
\end{proof}

\begin{setup}\label{set}
    Let $R$ be a Gorenstein local ring, and let $M$ be a CM $R$-module of codimension $g$. Let $\mathbb{F}: \dots\rightarrow F_1 \rightarrow F_0\xlongrightarrow{\epsilon} M\rightarrow0 $ be a minimal $R$-free resolution of $M$ along with an augmentation $\epsilon$. Let $\fa$ be a Gorenstein perfect ideal of grade $g$ such that $\fa \subseteq \ann_R(M)$. Set $\overline{(-)}:= (-)\otimes_R R/\fa$. Observe that $\pd_R(\overline{F_0})=\pd_R(R/\fa)=g$ and $\overline{F_0}$ is CM. Let $\mathbb{K}: 0\rightarrow K_g\rightarrow\dots\rightarrow K_1 \rightarrow K_0\rightarrow 0 $ be a minimal $R$-free resolution of $\overline{F_0}$.
    Let ${\Phi}: \mathbb{K} \longrightarrow\mathbb{F}$ be the chain map induced by $\overline{\epsilon} : \overline{F_0} \lra M$. Suppose the mapping cone of $\Phi^*$, as discussed in \ref{rem}, is given by
    \begin{equation}\label{equation}
    0\rightarrow F_0^*\rightarrow F_1^*\oplus K_0^*\rightarrow \dots \rightarrow  F_g^*\oplus K_{g-1}^* \xlongrightarrow{\psi} F_{g+1}^*\oplus K_{g}^*\xlongrightarrow{\xi} F_{g+2}^* \rightarrow\cdots.
\end{equation}
\end{setup}

\begin{remark}
    With {\rm Setup~\ref{set}}, as $\fa\subseteq \ann (M)$, it follows that $\overline{F_1} \rightarrow \overline{F_0} \rightarrow M \rightarrow 0$ is a minimal $\overline{R}$-free presentation of $M$. Applying $\Hom_{\overline{R}}(-,\overline{R})$, it yields the following commutative diagram of exact sequences:
    \begin{equation*}
    \xymatrixrowsep{1.4em}
\xymatrixcolsep{0.42em}
\xymatrix{
0 \ar[rr] && \Hom_{\overline{R}}(M,\overline{R}) \ar[rr] && \Hom_{\overline{R}}(\overline{F_0},\overline{R}) \ar[rr] \ar@{->>}[rd] && \Hom_{\overline{R}}(\overline{F_1},\overline{R}) \ar[rr] && \Tr_{\overline{R}}(M) \ar[rr] && 0\\
         &    &&     && \lambda_{\overline{R}}(M) \ar@{^{(}->}[ru].
}
\end{equation*}
Thus, there is a short exact sequence \begin{equation}\label{e}
    0 \to \Hom_{\overline{R}}(\overline{M},\overline{R}) \longrightarrow \Hom_{\overline{R}}(\overline{F_0},\overline{R}) \longrightarrow \lambda_{\overline{R}}(M) \to 0.
\end{equation}
\end{remark}

In the proof of our main result, the mapping cone $\cone(\Phi^*)$ in Setup~\ref{set} plays an important role. So we study some properties of $\cone(\Phi^*)$.

\begin{lemma}\label{lem}
   With {\rm Setup~\ref{set}}, the following hold true.
   \begin{enumerate}[\rm (1)]
       \item $H^{i}(\cone(\Phi^*)) = 0$ for all $i \neq g$.
       \item $H^g(\cone(\Phi^*)) \cong \lambda_{\overline{R}}(M)$, i.e., $\Ker(\xi)/\Image(\psi) \cong \lambda_{\overline{R}}(M)$, where $\psi$ and $\xi$ are the differentials described in {\rm\eqref{equation}}.
   \end{enumerate}
\end{lemma}

\begin{proof}
 The mapping cone of $\Phi^*$ gives a short exact sequence of complexes
 \begin{equation}
     0 \rightarrow \mathbb{K^*} \longrightarrow \cone(\Phi^*)\longrightarrow \mathbb{F^*}[-1] \rightarrow 0,
 \end{equation}
 which induces a long exact sequence of $R$-modules
 \begin{equation}\label{eq}
     \dots\rightarrow H^i(\mathbb{F^*}) \longrightarrow H^i(\mathbb{K^*}) \longrightarrow H^i(\cone \Phi^*) \longrightarrow H^{i+1}(\mathbb{F^*})\rightarrow\cdots.
 \end{equation}
Since $M$ is CM of codimension $g$, and $R$ is Gorenstein, it follows that $H^i(\mathbb{F^*})=\Ext_R^i(M,R)=0$ for all $~ i\neq g$, cf.~\cite[3.3.10]{BH98}.
For the same reason, $H^i(\mathbb{K^*}) = \Ext_R^i(\overline{F_0},R) = 0$ for all $i\neq g$. Hence, in view of \eqref{eq}, $H^i(\cone\Phi^*) = 0$ for all $i \neq g, g-1$. By \cite[p.~614, Lem.~11]{MS04}, there is a functorial isomorphism $\Ext_R^g(-,R) \cong \Hom_{\overline{R}}(-,\overline{R})$ on the category of $\overline{R}$-modules. In particular, one has that $H^g(\mathbb{F^*}) = \Ext_R^g(M,R) \cong \Hom_{\overline{R}}(M,\overline{R})$ and $H^g(\mathbb{K^*}) = \Ext_R^g(\overline{F_0},R)
 \cong \Hom_{\overline{R}}(\overline{F_0},\overline{R})$. Moreover, in view of \eqref{eq} and \eqref{e}, there exists a commutative diagram with exact rows:
\begin{equation*}
\xymatrixrowsep{1.7em}
\xymatrixcolsep{0.42em}
\xymatrix{
 0 \ar[rr] && H^{g-1}(\cone \Phi^*) \ar[rr] && \Ext_R^g(M,R) \ar[rr] \ar[d]^{\cong} && \Ext_R^g(\overline{F_0},R)\ar[d]^{\cong} \ar[rr] && H^g(\cone \Phi^*)\ar[d] \ar[rr] && 0\\
            &&  0\ar[rr]   && \Hom_{\overline{R}}(M,\overline{R}) \ar[rr] && \Hom_{\overline{R}}(\overline{F_0},\overline{R}) \ar[rr] && \lambda_{\overline{R}}(M) \ar[rr] &&0.
}
\end{equation*}
It follows that $H^{g-1}(\cone\Phi^*)=0$, and hence $H^g(\cone \Phi^*)\cong \lambda_{\overline{R}}(M)$ by the Snake Lemma. This completes the proof.
\end{proof}

\begin{lemma}\label{lmm}
With {\rm Setup~\ref{set}}, the following hold true.
   \begin{enumerate}[\rm (1)]
       \item $\pd_R(\coker\psi)$ is finite.
       \item $\Image\xi$ is zero or $($non-zero$)$ MCM depending on whether $\pd_R(M)$ is finite or infinite, respectively. The same holds for $\coker\xi$.
   \end{enumerate}
\end{lemma}

\begin{proof}
    By virtue of Lemma~\ref{lem}.(1), the complex \eqref{equation} of free $R$-modules yields the following exact sequences:
    \begin{align}
        &0\rightarrow F_0^*\rightarrow F_1^*\oplus K_0^*\rightarrow \dots \rightarrow  F_g^*\oplus K_{g-1}^* \xlongrightarrow{\psi} F_{g+1}^*\oplus K_{g}^* \rightarrow \coker\psi \to 0 \label{free-reso-coker-psi}\\
        &\mbox{and } \quad 0\rightarrow \Image{\xi}\longrightarrow F_{g+2}^* \longrightarrow F_{g+3}^*\longrightarrow F_{g+4}^*\rightarrow \cdots.\label{image-xi-syz}
    \end{align}
    The first assertion is now obvious from \eqref{free-reso-coker-psi}. For the second assertion, note that $M$ is CM of codimension $g$. So, if $\pd_R(M)$ is finite, then by the Auslander-Buchsbaum Formula, $\pd_R(M)=g$, hence $F_i=0$ for all $i>g$, which implies that $\Image{\xi}=0$.
    Conversely, let $\Image\xi=0$. In view of \ref{rem} (particularly, \eqref{diff-cone-phi-star}), note that
    $\xi = d^{g}_{\cone(\Phi^*)} = \begin{bmatrix}
    -(d^F_{g+2})^* &  0 \\
    0 & 0 
    \end{bmatrix}$.
    Since $\xi$ is a map from a free module to another free module, one gets that $d^F_{g+2}=0$, which yields that $\pd_R(M)$ is finite. Thus $\Image{\xi}=0$ \iff $\pd_R(M)$ is finite. When $\pd_R(M)$ is infinite, one has that $F_i\neq 0$ for all $i\ge 0$, and it follows from \eqref{image-xi-syz} and the Depth Lemma that $\Image{\xi}$ is MCM. Since $d^F_{g+2}$ is a map in a minimal free resolution, observe that $\coker{\xi}=0$ \iff $F_{g+2}^*=0$, which is equivalent to that $\pd_R(M)$ is finite. When $\coker{\xi}\neq 0$, by \eqref{image-xi-syz} and the Depth Lemma, $\coker{\xi}$ is MCM.   
\end{proof}

Next we show that with {\rm Setup~\ref{set}}, every $R$-module linked to $M$ by $\fa$ has the same projective (resp., injective) complexity and curvature as that of $\Image{\xi}$.

\begin{lemma}\label{lem:cx-curv-equa}
With {\rm Setup~\ref{set}}, assume that $N$ is an $R$-module which is linked to $M$ by $\fa$. Then there exists a short exact sequence $0\rightarrow N \rightarrow \coker{\psi}\rightarrow \Image{\xi} \rightarrow 0$. Moreover, the following statements hold true.
   \begin{enumerate}[\rm (1)]
        \item $\cx_R(N)= \cx_R(\Image{\xi}) $ and $\curv_R(N)= \curv_R(\Image{\xi}) $.
        \item $\injcx_R(N)= \injcx_R(\Image{\xi}) $ and $\injcurv_R(N)= \injcurv_R(\Image{\xi}) $.
    \end{enumerate}
\end{lemma}

\begin{proof}
In the mapping cone \eqref{equation} of $\Phi^*$, the map $\xi$ induces a short exact sequence 
 \begin{equation}{\label{cx}}
     0 \rightarrow \frac{ \Ker{\xi}}{\Image{\psi} }\longrightarrow \frac{F_{g+1}^*\oplus K_{g}^*}{\Image{\psi} }\longrightarrow \Image{\xi} \rightarrow 0.
 \end{equation}
 By virtue of Lemma~\ref{lem}.(2), $\Ker{\xi}/\Image{\psi} \cong \lambda_{\overline{R}}(M) \cong N$. Hence \eqref{cx} yields a short exact sequence $0\rightarrow N \rightarrow \coker{\psi}\rightarrow \Image{\xi} \rightarrow 0$. Note that $\pd_R{(\coker{\psi})}$ is finite (by Lemma~\ref{lmm}.(1)), and $R$ is Gorenstein. It follows that $\id_R{(\coker{\psi})}$ is finite. Therefore, the desired equalities are obtained from Remark~\ref{rmk:ses-cx}.
\end{proof}

The modules $M$ and $(\Image{\xi})^*$ have the same complexity (resp., curvature).

\begin{lemma}\label{lemm}
   With {\rm Setup~\ref{set}}, the following hold:
    \begin{center}
        $\cx_R(M)= \cx_R\big((\Image{\xi})^*\big)$ \quad and \quad $\curv_R(M)= \curv_R\big( (\Image{\xi})^* \big)$.
    \end{center}
\end{lemma}

\begin{proof}
In view of Lemma~\ref{lmm}.(2), if $\pd_R(M)$ is finite, then $\Image{\xi}=0$, hence the result follows from Remark~\ref{remark}.(1). So we may assume that $\pd_R(M)$ is infinite. Then, by Lemma~\ref{lmm}.(2), $\Image{\xi}$ and $\coker\xi$ are non-zero and MCM. Hence, by Lemma~\ref{lemma}, $\Image(\xi^*) \cong (\Image\xi)^*$.

By the discussion made in \ref{rem}, note that $d^{\cone(\Phi)}_n= \big(d^{n-2}_{\cone(\Phi^*)}\big)^*$. Thus, dualizing \eqref{equation}, one has the mapping cone of $\Phi$, which is
\begin{equation}\label{cone-phi}
    \dots \rightarrow F_{g+3} \rightarrow F_{g+2} \xlongrightarrow{\xi^*} F_{g+1}\oplus K_g \xlongrightarrow{\psi^*} F_g\oplus K_{g-1} \rightarrow \dots \rightarrow F_0 \rightarrow 0.
\end{equation}
Consider the short exact sequence $0\rightarrow \mathbb{F} \rightarrow \cone\Phi \rightarrow \mathbb{K}[-1] \rightarrow 0$ of chain complexes, see, e.g., \cite[1.5.2]{We94}. This induces a long exact sequence
$$\dots\rightarrow H_{n+1}(\mathbb{F}) \longrightarrow H_{n+1}(\cone \Phi) \longrightarrow H_n(\mathbb{K}) \longrightarrow H_n(\mathbb{F}) \rightarrow\cdots.$$
Hence, since $\mathbb{F}$ and $\mathbb{K}$ are acyclic,
it follows that $H_{n}(\cone \Phi) = 0$ for all $ n\ge2$.
So, in particular, the complex \eqref{cone-phi} yields an exact sequence
\begin{equation}\label{res}
   \dots\longrightarrow F_{g+4}\longrightarrow F_{g+3} \longrightarrow F_{g+2}\longrightarrow \Image{\xi^*}\rightarrow 0. 
\end{equation}
Note that \eqref{res} comes from $\mathbb{F}$, which is a minimal free resolution of $M$. So \eqref{res} gives a minimal free resolution of $\Image{\xi^*}$. Thus $ \beta_n^R{(\Image{\xi^*})} = \beta_{n+g+2}^R(M) $ for all $ n \ge 0 $. It follows that $ \cx_R(M) = \cx_R(\Image{\xi^*}) $ and $\curv_R(M) = \curv_R(\Image{\xi^*})$. Hence, since $\Image(\xi^*) \cong (\Image\xi)^*$, one obtains the desired equalities.
\end{proof}

Now we are in a position to prove our main results.


\begin{proof}[Proof of Theorem~\ref{thm:main}]
In view of \cite[p.~616, Cor.~15]{MS04}, the $R$-module $M$ is CM \iff $N$ is CM. So we may assume that both $M$ and $N$ are CM. Denote $g:=\htt(\fa)$. Then, by \cite[p.~616, Lem.~14]{MS04}, $\grade_R(M)=g$. Since $R$ is Gorenstein and $M$ is CM, in view of \cite[3.3.10]{BH98}, $\codim(M)=\grade_R(M)=g$. For the same reason, $\codim(N)=g$. Thus, we may work with exactly the same setup as described in Setup~\ref{set}.

(1) 
First, assume that $\pd_R(M)$ is finite. Then $\cx_R(M) = \curv_R(M) = 0$. Moreover, by Lemma~\ref{lmm}.(2), $\Image{\xi}=0$, which implies that $N \cong \coker\psi$ (cf.~Lemma~\ref{lem:cx-curv-equa}). Thus, by Lemma~\ref{lmm}.(1), $\pd_R(N)$ is finite. Since $R$ is Gorenstein, it follows that $\id_R(N)$ is finite. Hence $\injcx_R(N) = \injcurv_R(N) = 0$, and one has the desired equalities. So we may assume that $\pd_R(M)$ is infinite. In this case, by Lemma~\ref{lmm}.(2), $\Image{\xi}$ is MCM. Therefore, by \cite[Prop.~3.15.(2)]{DGS24},
\begin{equation}\label{DGS-3.15}
    \injcx_R(\Image{\xi})= \cx_R(\Image{\xi})^*, \ \injcurv_R(\Image{\xi})=\curv_R(\Image{\xi})^*.
\end{equation}
On the other hand, by Lemma~\ref{lem:cx-curv-equa},
\begin{equation}\label{inj-cx-N-image-xi}
    \injcx_R(N)= \injcx_R(\Image{\xi}) \ \mbox{ and} \ \injcurv_R(N)= \injcurv_R(\Image{\xi}).
\end{equation}
As $\cx_R(M)= \cx_R(\Image{\xi})^*$ and $\curv_R(M)= \curv_R(\Image{\xi})^*$ (cf.~Lemma~\ref{lemm}), the desired equalities follow from \eqref{DGS-3.15} and \eqref{inj-cx-N-image-xi}.

 (2) Since $R$ is a complete intersection ring, by \cite[Thm.~II.(1)]{AvBu00}, one has that $\injcx_R(N)=\cx_R(N)$. Hence (1).(i) yields that $\cx_R(M)=\cx_R(N)$. Here both $\cx_R(M)$ and $\cx_R(N)$ are finite, cf.~\cite[Thm.~(1.3) and Thm.~(5.3)]{AGP97}. Therefore, by Remark~\ref{remark}.(3), $\curv_R(M),\curv_R(N)\in \{0,1\}$. So, the equality $\curv_R(M)=\curv_R(N)$ follows from the following implications:
 \[
    \curv_R(M)=0 \Longleftrightarrow \cx_R(M)=0 \Longleftrightarrow \cx_R(N)=0 \Longleftrightarrow \curv_R(N)=0.
 \]

 (3) The proof of (1) shows that if $\pd_R(M)$ is finite, then $\pd_R(N)$ is so. Changing the role of $M$ and $N$, it follows that $\pd_R(M)$ is finite \iff $\pd_R(N)$ is finite. Note that both $M$ and $N$ are CM of codimension $g$. So, whenever $\pd_R(M)$ and $\pd_R(N)$ are finite, they are equal to $g=\htt(\fa)$ by the Auslander-Buchsbaum formula.
 

(4) 
In view of the discussion made in the first paragraph of the proof, note that whenever $\cid_R(M)$ and $\cid_R(N)$ are finite, by \cite[Thm.~(1.4)]{AGP97}, they are equal to $\htt(\fa)$. Thus, as before, it suffices to show that if $\cid_R(M)$ is finite, then $\cid_R(N)$ is finite. So we assume that $\cid_R(M)$ is finite. Then $\cid_R(M)=\htt(\fa)=g$. It follows that $\cid_R(\Omega_n^R(M))=0$ for all $n\ge g$, cf.~\cite[Lem.~(1.9)]{AGP97}. In particular, following the notation as in \ref{rem} and Setup~\ref{set},
\begin{center}
    $\cid_R(\Image{d^F_{g+2}})=0$.
\end{center}
If $\pd_R(M)$ is finite, then $\pd_R(N)$ is finite by (3), and hence $\cid_R(N)$ is finite. So we may assume that $\pd_R(M)$ is infinite. Then $\coker(d^F_{g+2})$ is MCM. Therefore, by Lemma~\ref{lemma}, $(\Image{d^F_{g+2}})^*\cong\Image{\big((d^F_{g+2})^*\big)}$. Following \eqref{diff-cone-phi-star} and \eqref{equation}, note that
$\xi = \begin{bmatrix}
    -(d^F_{g+2})^* &  0 \\
    0 & 0 
\end{bmatrix}$. Thus
\begin{equation*}
    \Image{\xi}=\Image{\big((d^F_{g+2})^*\big)}\cong(\Image{d^F_{g+2}})^*.
\end{equation*}
Since $\cid_R(\Image{d^F_{g+2}})=0$, by \cite[Lem.~3.5]{BJ11}, $\cid_R(\Image{d^F_{g+2}})^*=0$, and hence $\cid_R(\Image{\xi})=0$. Next, consider the short exact sequence $0\rightarrow N \rightarrow \coker{\psi}\rightarrow \Image{\xi} \rightarrow 0$ as shown in Lemma~\ref{lem:cx-curv-equa}. Since $\pd_R(\coker{\psi})$ is finite (by Lemma~\ref{lmm}.(1)), it follows that $\cid_R(N)$ is finite, see, e.g., \cite[Lem.~3.3]{GS24}.
\end{proof}



\begin{proof}[Proof of Corollary~\ref{cor:main}]
    In view of Remark~\ref{rmk:linked-ideal-red}, $R/I$ and $R/J$ are linked by $\fa$ as $R$-modules. The complexities (resp., curvatures) of $I$ and $R/I$ are equal. Since $R$ is Gorenstein, by Remark~\ref{rmk:ses-cx}.(2), both $J$ and $R/J$ have the same injecitve complexity (resp., curvature). Hence, the corollary follows from Theorem~\ref{thm:main}.
\end{proof}
\section{Some applications and examples}\label{sec:applications}

In this section, we note a few applications of our main theorem. We also construct some examples that complement our results.

\begin{definition}\label{defn:Ulrich-Burch}
    An $R$-module $M$ is called:
    \begin{enumerate}[(1)]
        \item \label{defn:Ulrich} (\cite{BHU87}, \cite[Def.~2.1]{GTT15}) Ulrich if $M$ is non-zero CM and $e(M) = \mu(M)$.
        \item \label{defn:Burch} (\cite[Def.~3.1]{DK23}) Burch if $M$ is a submodule of an $R$-module $L$ such that
        \begin{center}
            $\fm(M:_L \fm)\ne \mathfrak{m}M$, \ i.e., \ $\fm(M:_L \fm)\nsubseteq \mathfrak{m}M$.
        \end{center}
    \end{enumerate}
\end{definition}
The notion of $\fm$-full ideals was originally defined by D.~Rees (unpublished).
\begin{definition}\label{defn:m-full-Burch-ideal}
    An ideal $I$ of $R$ is called:
    \begin{enumerate}[(1)]
        \item \cite[p.~102, Def.~4]{Wat87} 
        $\fm$-full if $(\fm I:_R x)= I$ for some $x\in \fm$.
        \item \cite[Def.~3.7]{CIST18} weakly $\fm$-full if $(\fm I : \fm) = I$.
        \item \cite[Def.~2.1]{DKT20} Burch if $I$ is Burch as a submodule of $R$, i.e., $\fm(I:_R \fm) \nsubseteq \fm I$.
    \end{enumerate}
\end{definition}

\begin{remark}
\begin{enumerate}[\rm (1)]
    \item 
    \cite[Prop.~2.2.(2)]{GTT15} If $k$ is infinite, then a CM $R$-module $M$ is Ulrich if and only if $\fm M = (\underline{x})M$ for some $M$-regular sequence $\underline{x}$.
    \item 
    From the definition, it is clear that every $\fm$-full ideal is weakly $\fm$-full. But the converse is not necessarily true, cf.~{\rm Example~\ref{exam:int-closed-is-not-preserved}}.
\end{enumerate}
\end{remark}

\begin{example}\label{exam:Burch}
\begin{enumerate}[\rm(1)]
    \item \cite[Prop.~4.6]{GS24} If $k$ is infinite, and $R$ is not a field, then any integrally closed ideal $I$ of $R$ such that $\depth(R/I)=0$ is Burch.
    \item \cite[3.4]{DK23} If $N$ is a submodule of some $R$-module $L$ such that $\fm N \ne0$, then $M := \fm N$ is a Burch submodule of $L$. So $\fm N$ is a Burch $R$-module. In particular, for an ideal $I$ of $R$, if $\fm I$ is non-zero, then $\fm I$ is a Burch ideal.
\end{enumerate}
\end{example}

\begin{para}\label{para:max-cx-curv}
It is well known that
\begin{align*}
    \cx_R(k)=\sup\{\cx_R(M)\}=\sup\{\injcx_R(M)\}=\injcx_R(k) \quad \mbox{and} \\
\curv_R(k)=\sup\{\curv_R(M)\}=\sup\{\injcurv_R(M)\}=\injcurv_R(k),
\end{align*}
where $M$ ranges over all (finitely generated) $R$–modules, cf.~\cite[Prop.~2]{Avr96}. Thus, an $R$-module $M$ is said to have maximal complexity (resp., curvature) if $\cx_R(M) = \cx_R(k)$ (resp., $\curv_R(M) = \curv_R(k)$). Similarly, modules of maximal injective complexity (resp., curvature) are defined.
\end{para}

\begin{para}\label{para:max-cx-curv-survey}
The following classes of modules have maximal complexity and curvature:
(1)~Each non-zero homomorphic image of a finite direct sum of
syzygies of the residue field \cite[Cor.~9]{Avr96}. (2)~Every Ulrich module over a CM local ring \cite[Thm.~3.8]{DG23}; see also Remark~\ref{rmk:Ulrich-inj-cx-curv}. (3)~$\fm$-primary integrally closed ideals \cite[Thm.~2.7]{GP23}. (4)~Integrally closed ideals $I$ of $R$ such that $\depth(R/I) = 0$ \cite[Cor.~5.5.(1)]{GS24}. (5)~Burch modules \cite[Rmk.~3.10]{DK23}. Among these results, note that (3) is contained in (4), while under mild conditions, (4) is contained in (5) as shown in \cite[Prop.~4.6]{GS24}. Note that Burch modules and ideals also have maximal injective complexity (resp., curvature); see \cite[Thm.~5.1.(1)]{GS24}.
\end{para}

\begin{remark}\label{rmk:Ulrich-inj-cx-curv}
    Let $M$ be an Ulrich module over a CM local ring $R$. In \cite[Thm.~3.8]{DG23}, it is shown that $M$ has maximal complexity and curvature. The same argument results in $M$ having maximal injective complexity and curvature.
\end{remark}

As an application of Theorem~\ref{thm:main}, we obtain the following.

\begin{corollary}\label{cor:Burch}
    Let $(R,\fm,k)$ be a Gorenstein local ring. Let $M$ and $N$ be two $R$-modules linked by a Gorenstein perfect ideal $\fa$. Suppose that $M$ is CM.
    \begin{enumerate}[\rm (1)]
        \item If $M$ is a non-zero homomorphic image of a direct sum of syzygy modules of $k$, then
            $\injcx_R(N)=\cx_R(k)$ and $\injcurv_R(N)=\curv_R(k)$.
       \item If $e(M)< 2\mu(M)$, then $\injcx_R(N)=\cx_R(k)$ and $\injcurv_R(N)=\curv_R(k)$.
        \item If $e(M)< 2\type(M)$, then $\cx_R(N)=\cx_R(k)$ and $\curv_R(N)=\curv_R(k)$.
        \item If $M$ is Burch $($cf.~{\rm Example~\ref{exam:Burch}}$)$, or if $ e(M) < 2 \min\{\mu(M), \type(M)\}$ $($e.g., if $M$ is Ulrich$)$, then
        \begin{enumerate}[\rm (i)]
            \item $\cx_R(N) = \injcx_R(N)=\cx_R(k)$.
            \item $\curv_R(N) = \injcurv_R(N)=\curv_R(k)$.
        \end{enumerate}
     \end{enumerate}
\end{corollary}

\begin{proof}
 By virtue of Theorem~\ref{thm:main}.(1), the statements in (1), (2) and (3) follow from \cite[Cor.~9]{Avr96}, \cite[Cor.~5.5.(4)]{DGS24} and \cite[Cor.~5.5.(6)]{DGS24} respectively.

(4) When $M$ is Burch, it has maximal projective (resp., injective) complexity and curvature; see \cite[Thm.~5.1.(1)]{GS24}. Hence, the result follows directly from Theorem~\ref{thm:main}.(1). If $M$ is Ulrich, then by definition, $e(M) = \mu(M)$, and hence \cite[Thm.~4.5]{DGS24} yields $e(M) = \type(M)$. In general, if $ e(M) < 2 \min\{\mu(M), \type(M)\}$, the equalities in (i) and (ii) can be obtained from (2) and (3).
\end{proof}

We finally note that in a Gorenstein local ring, an $\fm$-primary ideal which is linked to a Burch ideal also has maximal (injective) complexity and curvature.

\begin{corollary}\label{cor:Burch-ideal}
    Let $(R,\fm)$ be a Gorenstein local ring. Let $I$ and $J$ be two $\fm$-primary ideals of $R$ linked by a Gorenstein perfect ideal $\fa$. Suppose that $I$ is integrally closed, or $I$ is weakly $\fm$-full, or $I$ is Burch. Then $J$ has maximal projective $($resp., injective$)$ complexity and curvature.
\end{corollary}

\begin{proof}
    Note that the ring $R/I$ is CM of dimension zero. So, the result follows from \cite[Cor.~5.5.(1)]{GS24} and Corollary~\ref{cor:main}.
\end{proof}

It follows from \cite[Thm.~4.1.(a)]{MQS23} that Ulrich property of modules is preserved under linkage by a Gorenstein ideal.

\begin{proposition}\label{prop:Ulrich-linkage}
    Let $M$ and $N$ be two $R$-modules linked by a Gorenstein ideal $\mathfrak{a}$. Then $M$ is Ulrich \iff $N$ is so.
\end{proposition}

\begin{proof}
    Since $\mathfrak{a}M=0$, observe that $M$ is Ulrich as an $R$-module \iff $M$ is Ulrich as an $R/\mathfrak{a}$-module because depth, dimension, multiplicity and minimal number of generators of $M$ are preserved while passing from $R$ to $R/\mathfrak{a}$. Since $R/\mathfrak{a}$ is Gorenstein, replacing $R/\mathfrak{a}$ by $R$, we may assume that $R$ is Gorenstein, and the $R$-modules $M$ and $N$ are horizontally linked. Then, by \cite[p.~616, Lem.~14 and Cor.~15]{MS04}, $M$ is CM \iff $M$ is MCM \iff $N$ is MCM. Therefore, if $R$ is regular and $M$ is Ulrich, then $M$ is free, which implies that $N=0$, and hence $M$ must be zero, a contradiction. So we may assume that $R$ is not regular, i.e., $\fm$ is not a parameter ideal. Hence, the statement follows from \cite[Thm.~4.1.(a)]{MQS23}.
\end{proof}

Proposition~\ref{prop:Ulrich-linkage} is not true in general when $\mathfrak{a}$ is not a Gorenstein ideal.

\begin{example}\label{exam:Ulrich-is-not-preserved}
Let $R=k[[x,y]]$ over a field $k$. Set $\mathfrak{a}:=(x^2,xy)$, $M:=I/\mathfrak{a}$ and $N:=\fm/\mathfrak{a}$, where $I:=(x)$ and $\fm:=(x,y)$. Since $(\mathfrak{a} :_R I) = \fm$ and $(\mathfrak{a} :_R \fm) = I$, the modules $M$ and $N$ are linked by $\mathfrak{a}$. Here $\mathfrak{a}$ is not a Gorenstein ideal. Note that $M\cong k$ is an Ulrich $R$-module, while $N$ is not Ulrich because $e(N)=1\neq 2=\mu(N)$. In fact, $N$ is not a CM $R$-module as $\depth(N) = 0 \neq 1 = \dim(N)$.
\end{example}

A modification of Example~\ref{exam:Ulrich-is-not-preserved} ensures that an ideal (horizontally) linked to a Burch ideal is not necessarily Burch. The base ring in this example is not CM.

\begin{example}\label{exam:Burch-is-not-preserved}
Let $R=k[x,y]/(x^2,xy)$ over a field $k$. Then $R$ is not CM. Set $I:=(x)$ and $\fm:=(x,y)$. In view of Example~\ref{exam:Ulrich-is-not-preserved}, the ideal $I$ is (horizontally) linked to the maximal ideal $\fm$ of $R$. Here, $\fm$ is a Burch ideal. However, $I$ is not a Burch ideal because $\fm(I:_R \fm) = \fm I$.  
\end{example}

The following example ensures that even in a complete intersection local ring, an ideal $I$ linked to a Burch (resp., weakly $\fm$-full, or $\fm$-full) ideal $J$ by a Gorenstein ideal $\fa$ is not necessarily Burch (resp., weakly $\fm$-full, or $\fm$-full). In this example, $\fa$ is not perfect. It also shows that assertions (1) and (2) in Theorem~\ref{thm:main} do not hold in general if the ideal $\fa$ is not perfect.

\begin{example}\label{exam:Burch-m-full-is-not-preserved}
       Let $R=k[[x,y,z]]/(x^2,y^2)$ over a field $k$. Then $(R,\fm)$ is a local complete intersection ring of codimension $2$, where $\fm = (x,y,z)$. Set $\fa:=(xz, yz, z^2-xy)$, $I:=(x,z)$ and $J:=(x,yz,z^2)$. Then the following hold.
       \begin{enumerate}[(1)]
        \item 
         The ideals $I$ and $J$ are linked by $\fa$, where $\fa$ is Gorenstein, but $\fa$ is not perfect.
        \item 
        The ideal $J$ is Burch, while $I$ is not Burch. 
        \item 
        Although $J$ is $\fm$-full, $I$ is not even weakly $\fm$-full.
        \item 
        \begin{enumerate}[(i)]
            \item $\injcx_R(R/I)=\cx_R(R/I)=\cx_R(I)=\injcx_R(I)=1$.
            \item $\injcx_R(R/J) = \cx_R(R/J) = \cx_R(J) = \injcx_R(J)= 2$.
        \end{enumerate}
    \end{enumerate}
 \end{example}

\begin{proof}
    (1) Observe that $\fa \subseteq I \cap J$, $I= (\fa:_R J) $ and $J= (\fa:_R I)$. So, the ideals $I$ and $J$ are linked by $\fa$. Note that $R/\fa$ is Artinian, and the socle of $R/\fa$ is generated by $xy$. It follows that the ideal $\fa$ is Gorenstein. Taking the minimal free presentation $R^3\xlongrightarrow{[xz~yz~z^2-xy]}R\rightarrow R/\fa \rightarrow 0$ of $R/\fa$, since the map $[xz~yz~z^2-xy]$ is not injective, it follows that $\pd_R(R/\fa)\ge 2$. Therefore, since $\depth(R)=1$, by the Auslander-Buchsbaum formula, $\pd_R(R/\fa)=\infty$. Thus, $\fa$ is not perfect.

    (2) Since $\fm(I:_R \fm) = \fm^2 = (xy,xz,yz,z^2)= \fm I$, the ideal $I$ is not Burch. However, $J$ is a Burch ideal because $z^2\in \fm(J:_R\fm)$, but $z^2\notin \fm J$.

    (3) Note that $\fm J = (xy, xz, yz^2, z^3)$ and $(\fm J:_R z)= J$. Thus $J$ is $\fm$-full. However, $I$ is not even weakly $\fm$-full because $(\fm I:_R \fm)=\fm\neq I$.

    (4) Since $R$ is complete intersection, for any $R$-module $M$, one has $\cx_R(M)=\injcx_R(M)$, cf.~\cite[Thm.~II.(1)]{AvBu00}. So, it is enough to compute only $\cx_R(R/I)$ and $\cx_R(J)$. Since $Rx=\ann_R(x)$, one has $\beta_n^R(R/(x))=1$ for all $n\ge 0$. Consider the minimal free resolution $\mathbb{F}_x : ~ \cdots \to R \xrightarrow{x} R \xrightarrow{x} R \to 0 $ of $R/(x)$, and $\mathbb{F}_z: 0\rightarrow R \xrightarrow{z} R \rightarrow 0$ of $R/(z)$ as $z$ is $R$-regular. Since $R/(x,z) \cong R/(x) \otimes_R R/(z)$ and $\Tor_n^R(R/(x), R/(z))=0$ for all $n\ge 1$, the tensor product $\mathbb{F}_x\otimes_R\mathbb{F}_z$ yields a minimal R-free resolution of $R/(x,z)$. It follows that $\beta_n^R(R/(x,z))=2$ for all $n\ge 1$, and hence $\cx_R(R/I)=1$. Since $J$ is a Burch ideal, in view of \ref{para:max-cx-curv-survey}, $\cx_R(J)=\cx_R(k)=2$, where the second equality follows from \cite[Thm.~6]{Tat57} as $R$ is a complete intersection local ring of codimension $2$.
\end{proof}


The properties `integrally closed' and `$\fm$-full' are not always preserved under linkage of ideals by a regular sequence even in a regular local ring.

\begin{example}\label{exam:int-closed-is-not-preserved}
    Let $R=k[[x,y]]$ be a formal power series ring in two variables $x,y$ over a field $k$. Then $R$ is regular local. Set $\fm:=(x,y)$. Fix integers $m,n\ge 2$. Set $\fa := (x^m,y^n)$ and $I:=(x^m,x^{m-1}y^{n-1},y^n)$. Then, $I$ is linked to $\fm$ by the $R$-regular sequence $x^m,y^n$. Moreover, the following hold.
    \begin{enumerate}[(1)]
        \item Though $\fm$ is integrally closed, but $I$ is integrally closed only when $m+n\le 5$.        
        \item Both $\fm$ and $I$ are weakly $\fm$-full and Burch.
        \item  Though $\fm$ is $\fm$-full, but $I$ is $\fm$-full \iff $\min\{m,n\} = 2$.
    \end{enumerate}
\end{example}

\begin{proof}
    Since $(\fa :_R \fm) = I$ and $(\fa :_R I) = \fm$, the ideal $I$ is linked to $\fm$ by $\fa$.

    (1) Denote the integral closure of $I$ by $\overline{I}$. To compute $\overline{I}$, use \cite[1.4.6]{HS06}. Without loss of generality, assume that $m\ge n$. If $n\ge 3$, then $x^{m-1}y\in \overline{I}$, but $x^{m-1}y\notin I$. If $m\ge 4$ and $n=2$, then $x^{m-2}y\in \overline{I}$, but $x^{m-2}y\notin I$. It remains to consider the cases when $n=2$ and $m=2$ or $3$. In each of these two cases, $\overline{I}=I$. Thus, $I$ is integrally closed \iff $m+n \le 5$.
    
    (2)
    Observe that $x^{m-1}y^{n-1}\in \fm(I:_R\fm)$, but $x^{m-1}y^{n-1}\notin \fm I$. So $I$ is a Burch ideal. Since $\fm I=(x^{m+1},x^{m}y,xy^{n},y^{n+1})$, it follows that
    \begin{align*}
        (\fm I:_R\fm) & = (\fm I:_R x) \cap (\fm I:_R y) \\
        & = (x^{m},x^{m-1}y,y^{n})\cap (x^{m},xy^{n-1},y^{n}) \\
        & = (x^m,x^{m-1}y^{n-1},y^n) = I.
    \end{align*}
    Thus $I$ is weakly $\fm$-full.

   (3) From the definition, it is clear that $\fm$ is $\fm$-full. Note that $I$ is an $\fm$-primary ideal, and $\mu(I)=3$. Set $r:=\min\{m,n\}$. Then $I\subseteq\fm^r$, but $I\nsubseteq\fm^{r+1}$. Hence, by \cite[Thm.~4]{Wat87}, $I$ is $\fm$-full \iff $r=2$.
\end{proof}


We close this article by presenting the following natural questions.

\begin{question}\label{ques-1}
    Let $R$ be a Gorenstein local ring. Let $I$ and $J$ be two proper ideals of $R$ linked by a Gorenstein perfect ideal such that $R/I$ and $R/J$ are CM.\\
    (1)~Is $\cx_R(I) = \cx_R(J)$? \; (2)~Is $\curv_R(I) = \curv_R(J)$?\\
    (3)~If $I$ is Burch (resp., weakly $\fm$-full), then is $J$ Burch (resp., weakly $\fm$-full)?
    \mycomment{
    \begin{enumerate}[(1)]
        \item Is $\cx_R(I) = \cx_R(J)$?
        \item If $I$ is Burch (resp., weakly $\fm$-full, or $\fm$-full), then is $J$ the same?
    \end{enumerate}
    }
\end{question}

\begin{question}\label{ques-2}
    Let $R$ be a Gorenstein local ring. Let $M$ and $N$ be two CM $R$-modules linked by a Gorenstein perfect ideal.\\
    (1)~Is $\cx_R(M) = \cx_R(N)$? \; (2)~Is $\curv_R(M) = \curv_R(N)$?
\end{question}

In view of Theorem~\ref{thm:main}.(1), affirmative answers to Question~\ref{ques-2} ensure that both injective complexity and curvature of CM modules are also preserved under linkage by a Gorenstein perfect ideal over a Gorenstein local ring. By Theorem~\ref{thm:main}.(2), if $R$ is complete intersection, then Question~\ref{ques-2} has positive answers.

\subsection*{Acknowledgments}
This work forms part of the first-named author's Ph.D.~thesis. He gratefully acknowledges the financial support provided by the UGC, Ministry of Education, Government of India, during his Ph.D.~studies.

\bibliographystyle{plain}
\bibliography{mainbib}

\end{document}